\documentclass[letterpaper,11pt]{amsart}
\usepackage[margin=1.2in]{geometry}
\usepackage{amsmath,amsthm,amssymb}
\usepackage{xspace,xcolor}
\usepackage[breaklinks,colorlinks,citecolor=teal,linkcolor=teal,urlcolor=teal,pagebackref,hyperindex]{hyperref}
\usepackage[alphabetic]{amsrefs}
\usepackage[all]{xy}
\usepackage[english]{babel}
\usepackage{enumitem}
\usepackage{tikz}
\usepackage{tikz-cd}
\usepackage{mathrsfs}

\usepackage{mathtools}

\newcounter{intro}

\newtheorem{intro-conjecture}[intro]{Conjecture}
\newtheorem{intro-corollary}[intro]{Corollary}
\newtheorem{intro-theorem}[intro]{Theorem}

\newcommand{\theoremref}[1]{\hyperref[#1]{Theorem~\ref*{#1}}}
\newcommand{\lemmaref}[1]{\hyperref[#1]{Lemma~\ref*{#1}}}
\newcommand{\definitionref}[1]{\hyperref[#1]{Definition~\ref*{#1}}}
\newcommand{\propositionref}[1]{\hyperref[#1]{Proposition~\ref*{#1}}}
\newcommand{\conjectureref}[1]{\hyperref[#1]{Conjecture~\ref*{#1}}}
\newcommand{\corollaryref}[1]{\hyperref[#1]{Corollary~\ref*{#1}}}
\newcommand{\exampleref}[1]{\hyperref[#1]{Example~\ref*{#1}}}

\theoremstyle{plain}
\newtheorem{thm}{Theorem}[section]

\newtheorem{lem}[thm]{Lemma}
\newtheorem{prop}[thm]{Proposition}
\newtheorem{cor}[thm]{Corollary}

\theoremstyle{definition}
\newtheorem{defi}[thm]{Definition}

\newtheorem{eg}[thm]{Example}

\theoremstyle{remark}
\newtheorem{rmk}[thm]{Remark}

\def\N{{\mathbf N}}
\def\Z{{\mathbf Z}}
\def\Q{{\mathbf Q}}

\def\C{{\mathbf C}}

\def\P{{\mathbf P}}

\def\cA{\mathcal{A}}

\def\cC{\mathcal{C}}
\def\cD{\mathcal{D}}
\def\cE{\mathcal{E}}

\def\cH{\mathcal{H}}

\def\cL{\mathcal{L}}
\def\cM{\mathcal{M}}
\def\cN{\mathcal{N}}
\def\cO{\mathcal{O}}

\def\cT{\mathcal{T}}

\def\.{\cdot}
\def\^{\widehat}

\def\de{\partial}

\def\({\left(}
\def\){\right)}

\renewcommand{\and}{ \ \ \text{ and } \ \ }

\begin{document}

\author[B.~Dirks]{Bradley Dirks}

\address{Department of Mathematics, Stony Brook University, Stony Brook, NY 11794-3651, USA}

\email{bradley.dirks@stonybrook.edu}

\thanks{The author was partially supported by the National Science Foundation under Grant MSPRF DMS-2303070.}

\title[Relative Du Bois complexes via MHMs]{A Hodge module construction of relative Du Bois complexes}

\begin{abstract} We construct relative Du Bois and relative intersection Du Bois
complexes for smooth-factorizable morphisms to a smooth complex base using relative de
Rham complexes of mixed Hodge modules and Grothendieck duality. For a
smooth morphism, these complexes recover the relative K\"ahler
differentials. Comparisons with the relative Du Bois complexes of Kov\'{a}cs and Ning remain open.

We construct finite
filtrations on the absolute Du Bois and intersection Du Bois complexes
whose graded pieces are the corresponding relative complexes tensored
with differential forms from the base. We also study restriction to
fibers and its relation to mild Hodge-theoretic singularity classes.
\end{abstract}

\maketitle

\section{Introduction}

Let $f\colon X\to B$ be a smooth morphism of smooth complex varieties, with $B$ connected.
The exact sequence
\[
0\longrightarrow f^*\Omega_B^1\longrightarrow \Omega_X^1
\longrightarrow \Omega_{X/B}^1\longrightarrow 0
\]
induces, for every $p$, a finite filtration on $\Omega_X^p$ with
graded pieces
\[
\operatorname{Gr}_{-j}^G\Omega_X^p
\simeq
\Omega_{X/B}^{p-j}\otimes_{\mathcal O_X}f^*\Omega_B^j.
\]

Moreover, in this smooth setting the objects $\Omega_{X/B}^j$ satisfy the base change property: for any $b\in B$, let $i_b \colon X_b \to X$ be the inclusion of the fiber over $b$. Then for all $j\in \Z$, by \cite{Hartshorne}*{Prop. 8.10} we have isomorphisms
\[ i_b^*(\Omega_{X/B}^j) \cong \Omega_{X_b/\C}^j.\]

The purpose of this paper is to construct objects with similar properties for smooth-factorizable morphisms $f\colon X \to B$ when $X$ is singular or $f$ is not smooth. This problem was originally studied by Kov\'{a}cs \cite{KovacsFirstRelDB}, who constructed the individual graded pieces $\underline{\Omega}_{X/B}^p$. The construction was recently refined by Kov\'{a}cs and Taji \cite{KovacsTaji}, when $\dim B =1$, to a filtered complex whose associated graded pieces agree with Kov\'{a}cs' complexes. 

In this work, we provide alternative constructions of the individual graded piece objects $\underline{\Omega}_{X/B}^p$ which are closely related to the author's forthcoming joint work with Chen and Olano \cite{CDOHigher}. We pose the question of comparing our objects with those of Kov\'{a}cs (and other, related constructions due to Ning \cite{NingDeformations}). See Section \ref{sec-OpenProblems}.

A \emph{smooth-factorizable} morphism is $f\colon X\to B$, with $B$ smooth and connected, admitting a factorization
\[
X\xrightarrow[]{\iota} B\times Y\longrightarrow B
\]
with $Y$ smooth and $\iota$ a closed embedding. We call $f$ locally smooth-factorizable if $B$ admits an open cover $\{U_\alpha\}_{\alpha \in \Lambda}$ such that $f\vert_{f^{-1}(U_\alpha)} \colon f^{-1}(U_\alpha) \to U_\alpha$ is smooth-factorizable for all $\alpha \in \Lambda$. For example, any projective morphism is locally smooth-factorizable.

\begin{defi} Let $X \xrightarrow[]{\iota} B \times Y \to B$ be a smooth factorization of $f \colon X \to B$. For any $p\in \Z$, define
\[ \underline{\Omega}_{X/B}^p \coloneqq \mathbb D_X({\rm Gr}^F_{p} {\rm DR}_{B \times Y/B}(\iota_* \mathbf D_X(\Q_X^H)))[p] \]
which, by construction, lies in $D^b_{\rm coh}(\cO_X)$.

We also define the relative intersection Du Bois complex
\[ {\rm I}\underline{\Omega}_{X/B}^p \coloneqq \mathbb D_X({\rm Gr}^F_{p}{\rm DR}_{B \times Y/B}(\iota_* \mathbf D({\rm IC}_X^H)))[p-\dim X]\]
so that there is a morphism for all $p \in \Z$:
\[ \underline{\Omega}_{X/B}^p \to {\rm I}\underline{\Omega}_{X/B}^p.\]
\end{defi}

The duality is essential. We want to define an object which, for smooth morphisms, agrees with the relative K\"{a}hler differentials as at the start of the introduction. The naive relative de Rham complex construction does not satisfy this property.

\begin{rmk} If $f \colon X \to B$ is smooth with $\dim B > 0$, then in particular, it is smooth-factorizable via the graph construction. A computation shows that, in this case,
\[ {\rm I}\underline{\Omega}_{X/B}^p = \underline{\Omega}_{X/B}^p = \Omega^p_{X/B}.\] 
See Example \ref{eg-SmoothMaps}.

In fact, this relation \emph{fails} if we use the naive definition (without the ``double dual'') of relative Du Bois complex via mixed Hodge modules. Indeed, if we consider
\[ {\rm Gr}^F_{-p} {\rm DR}_{B\times Y/B}(\iota_* \Q_X^H)[p],\]
then in this example, we have
\[ {\rm Gr}^F_{-\dim X} {\rm DR}_{B\times X/B}(\iota_* \Q_X^H[\dim X]) \cong \omega_X \neq \Omega_{X/B}^{\dim X} = 0,\]
because ${\rm Gr}^F_{-\dim X} \iota_* \Q_X^H$ is the first non-zero piece of the Hodge filtration.

This explains the duality in the defining formula for the relative Du Bois complexes, and shows that, in general, the relative de Rham functor does not satisfy a duality compatibility relation, unlike the usual de Rham functor.
\end{rmk}

Our first main result is the existence of a natural filtration (in the sense of derived categories) of the absolute Du Bois complex of $X$ by the relative Du Bois complexes. The proof we give uses smooth-factorizability in an essential way.

\begin{intro-theorem} \label{thm-filtration} Let $f\colon X \to B$ be a smooth-factorizable morphism with $B$ smooth and connected. For any $p$, there is a commutative diagram
\[ \begin{tikzcd}  0 = G_{-\dim B-1} \underline{\Omega}_X^p \ar[r] \ar[d] & G_{-\dim B} \underline{\Omega}_X^p \ar[r] \ar[d] & G_{-(\dim B -1)} \underline{\Omega}_X^p \ar[r] \ar[d] & \ldots \ar[r] & G_{0} \underline{\Omega}_X^p = \underline{\Omega}_X^p \ar[d]  \\  0 = G_{-\dim B-1} {\rm I} \underline{\Omega}_X^p \ar[r] & G_{-\dim B} {\rm I}\underline{\Omega}_X^p \ar[r]& G_{-(\dim B -1)} {\rm I}\underline{\Omega}_X^p \ar[r]  & \ldots \ar[r] & G_{0} {\rm I}\underline{\Omega}_X^p = {\rm I}\underline{\Omega}_X^p ,\end{tikzcd}\]
in $D^b_{\rm coh}(\cO_X)$ such that
\[ {\rm Gr}^G_{-j} \underline{\Omega}_X^p \coloneqq {\rm cone}(G_{-j-1}\underline{\Omega}_X^p \to G_{-j}\underline{\Omega}_X^p)  \simeq \underline{\Omega}_{X/B}^{p-j} \otimes_{\cO_X} f^*(\Omega_B^j),\]
\[ {\rm Gr}^G_{-j} {\rm I}\underline{\Omega}_X^p \coloneqq {\rm cone}(G_{-j-1}{\rm I}\underline{\Omega}_X^p \to G_{-j}{\rm I}\underline{\Omega}_X^p)\simeq {\rm I}\underline{\Omega}_{X/B}^{p-j} \otimes_{\cO_X} f^*(\Omega_B^j).\]
\end{intro-theorem}

Thus, our relative Du Bois complexes filter the absolute Du Bois complexes in an analogous way to the case of a smooth morphism. If $f\colon X \to B$ is smooth, this is the filtration from the start of the paper.

\begin{rmk} For $\dim B = 1$, the two-step filtration and identification of the associated graded pieces are equivalent to the exact triangle for every $p\in \Z$:
\[ \underline{\Omega}_{X/B}^{p-1} \otimes f^*(\Omega_B^1)  \to \underline{\Omega}_X^p \to\underline{\Omega}_{X/B}^p \xrightarrow[]{+1}.\]
\end{rmk}

Our remaining results concern base-change. Let $f\colon X \to B$ be a flat morphism which is locally smooth-factorizable, in the sense that, for all $b\in B$, there is a neighborhood $b\in U \subseteq B$ such that $f^{-1}(U) \to U$ is smooth-factorizable. Then for all $b\in B$ and $p\in \Z$, we construct base-change morphisms
\[L i_b^* \underline{\Omega}_{f^{-1}(U)/U}^p \to \underline{\Omega}_{X_b}^p\]
where $X_b = f^{-1}(b)$ is the fiber over $b$. These base-change morphisms are compatible with restricting the open neighborhood $U$. For the relative intersection Du Bois complexes, there is no canonical comparison morphism, but in some favorable situations, we can still compare the objects.

Given $f\colon X \to B$, we let $i_b \colon X_b = f^{-1}(X) \hookrightarrow X$ denote the inclusion of the fiber over $b\in B$. If $b\in U$ is an open neighborhood, we also write $i_b \colon X_b \hookrightarrow f^{-1}(U)$.

Our first base-change result is ``generic on the base'':
\begin{intro-theorem} \label{thm-generalBC} Let $f\colon X \to B$ be a locally smooth-factorizable morphism with $B$ a smooth, connected variety. Then there exists a dense open subset $U \subseteq B$ so that $f^{-1}(U) \to U$ is flat and smooth-factorizable and such that, for all $b\in U$ and $p\in \Z$, the base-change morphism
\[ L i_b^* \underline{\Omega}_{f^{-1}(U)/U}^p \to \underline{\Omega}_{X_b}^p\]
is a quasi-isomorphism.

Moreover, for all $p\in \Z$ and all $b\in U$, there exists a quasi-isomorphism
\[  L i_b^* {\rm I}\underline{\Omega}_{f^{-1}(U)/U}^p \simeq {\rm I}\underline{\Omega}_{X_b}^p.\]
\end{intro-theorem}

Thus, our object has the correct value on fibers over general points of $B$. This should be compared with the analogous result \cite{JiKovacs}. 

The generic base change does not control base change to a special fiber. Under the mild singularity notions of \cite{CDOHigher}, however, the same compatibility persists in a neighborhood of a mildly singular fiber. This result was our original motivation for this construction.

\begin{intro-theorem} \label{thm-higherBC} Let $f\colon X \to B$ be flat and projective (in particular, locally smooth-factorizable) with $B$ smooth and connected.

If a fiber $X_{b_0}$ is symbolically $m$-Du Bois, then there exists a neighborhood $b_0 \in U$ over which $f$ is smooth-factorizable and such that the natural base-change morphism
\[ L i_b^* \underline{\Omega}_{f^{-1}(U)/U}^p \to \underline{\Omega}_{X_b}^p\]
is a quasi-isomorphism for all $p\leq m$ and all $b\in U$.

If $X_{b_0}$ is instead weakly $m$-rational, then there exists a neighborhood $b_0 \in U$ over which $f$ is smooth-factorizable and a quasi-isomorphism
\[ L i_b^* {\rm I}\underline{\Omega}_{f^{-1}(U)/U}^p \simeq {\rm I}\underline{\Omega}_{X_b}^p\]
for all $p\leq m$ and all $b\in U$.
\end{intro-theorem}

Related base-change results were obtained in the recent preprint \cite{NingDeformations}, though the mild class of singularities considered there is different from ours. Moreover, Ning considers a notion of relative Du Bois complex which differs from Kov\'{a}cs' original construction, and which seems better suited to base change results. We ask, in Section \ref{sec-OpenProblems}, whether our construction is comparable to Ning's.

\noindent {\bf Conventions.} Throughout, a variety is a reduced separated scheme of finite type over $\C$. We assume that the fibers of any morphism are reduced. Whenever the intersection complex or weakly $m$-rational singularities are considered, the variety in question is assumed equidimensional.

For a variety $X$, we write $d_X\coloneqq \dim X$ for convenience of notation.

We let $\mathbb D_X(-) \coloneqq R\cH om_{\cO_X}(-,\omega_X^\bullet)$ denote the Grothendieck duality functor on a (possibly singular) variety $X$. We will often make use of the natural isomorphism
\[ \mathbb D_Z \circ \iota_*(-) \simeq \iota_* \circ \mathbb D_X(-)\]
below.

Grothendieck duality relates the standard and shriek restriction functors: for a closed embedding of smooth varieties $i \colon Y \to Z$,
\[ \mathbb D_Y \circ Li^!(-) \simeq Li^*(-) \circ \mathbb D_Z\]
where $Li^!(-) \coloneqq R \cH om_{\cO_Z}(\cO_Y,-) \simeq Li^*(-) \otimes \det \cN_{Y/Z} [\dim Y - \dim Z]$, where $\cN_{Y/Z}$ is the normal bundle of the closed embedding.

We let $\mathbf D(-)$ denote the duality for filtered $\cD_Y$-modules on a smooth variety $Y$, as defined in \cite{SaitoMHP}*{Sec. 2.4}.

For a filtered object $(\cM,F)$, the notation $F[j]$ denotes the shifted filtration:
\[ (F[j])_p \cM \coloneqq F_{p-j} \cM.\]

\noindent {\bf Acknowledgments.} The author thanks Qianyu Chen, S\'{a}ndor Kov\'{a}cs, Haoming Ning, Sebasti\'{a}n Olano and Sung Gi Park for conversations related to this work.

\section{Relative de Rham complexes}\label{sec-prelim}
We review some basic constructions for relative de Rham complexes of filtered $\cD$-modules in this section.

\subsection{Relative Spencer Complex}
Let $\pi \colon Z \to B$ be a smooth morphism of smooth varieties of relative dimension $r$, and let $(\cM,F)$ be a filtered right $\cD_Z$-module. We have the relative de Rham complex ${\rm DR}_{Z/B}(\cM,F)$ written out in cohomological degrees $-r,\ldots, 0$:
\begin{equation} \label{eq-relDR}  [(\cM,F[r]) \otimes_{\cO} \bigwedge^{r} \cT_{Z/B} \xrightarrow[]{\nabla} (\cM,F[r-1]) \otimes_{\cO} \bigwedge^{r-1} \cT_{Z/B} \xrightarrow[]{\nabla} \dots \xrightarrow[]{\nabla} (\cM,F)].\end{equation}

Here $\nabla$ is the usual right relative Spencer differential. The $p$th Hodge piece (meaning $F_p$) of the $-k$th term is
\[
F_{p-k}\mathcal M\otimes\bigwedge^kT_{Z/B}.
\]

Thus the differential preserves $F_\bullet$, and the associated graded is a complex of $\mathcal O_Z$-modules. For bounded complexes we use the totalization of the obvious double complex. This construction is functorial and commutes with restriction along open immersions over $B$.

We now specialize to $Z=B\times Y$ for $Y$ a smooth $r$-dimensional variety with $q\colon Z \to Y$ the second projection. In this case
$\cT_{Z/B}=q^*\cT_Y$, and the product decomposition
\[
\cT_Z \simeq\pi^*\cT_B\oplus\cT_{Z/B}
\]
allows us to factor the absolute Spencer complex into the relative
and base directions. Indeed, each term of the complex ${\rm DR}_{Z/B}(\cM,F)$ has the structure of a right $\pi^{-1}(\cD_B)$-module, and the morphisms are right $\pi^{-1}(\cD_B)$-linear. Thus, using the construction of the usual filtered de Rham complex for $\cD_B$-modules, we have a natural isomorphism
\begin{equation} \label{eq-iterateDR} {\rm DR}_Z(\cM,F) \simeq {\rm DR}_{\pi^{-1}(\cD_B)}{\rm DR}_{Z/B}(\cM,F).\end{equation}

This extends immediately to bounded complexes of filtered $\cD_Z$-modules $(\cM^\bullet,F)$. 

We have an analogous factorization for the associated graded. Set
\[\cA_B \coloneqq {\rm Sym}_{\cO_B}\cT_B, \qquad \cA_Z \coloneqq {\rm Sym}_{\cO_Z}\cT_Z,\qquad\cA_{Z/B} \coloneqq {\rm Sym}_{\cO_Z}\cT_{Z/B}.\]

The product decomposition
\[\cT_Z\simeq\pi^*\cT_B\oplus\cT_{Z/B}\]
induces an isomorphism of graded $\cO_Z$-algebras
\[\cA_Z\simeq\pi^*(\cA_B) \otimes_{\cO_Z}\cA_{Z/B}.\]
In particular, there are natural quotient maps
\[\cA_Z\to\pi^*(\cA_B) \to\cO_Z,\]
where the first kills the symbols of the fiber-direction tangent vectors and the second
kills the symbols of tangent vectors coming from the base $B$.

For a bounded complex $(\cM^\bullet,F)$, put
\[
\cN^\bullet\coloneqq{\rm Gr}^F\cM^\bullet
=\bigoplus_a{\rm Gr}_a^F\cM^\bullet.
\]
The graded relative Spencer description (see \cite{SaitoMHP}*{Lem. 2.1.6}) gives natural
quasi-isomorphisms
\begin{equation}\label{eq-relative-symbol-tensor}
\bigoplus_p{\rm Gr}_p^F{\rm DR}_{Z/B}(\cM^\bullet)
\simeq
\cN^\bullet\otimes_{\cA_Z}^L\pi^*(\cA_B)
\simeq
\cN^\bullet\otimes_{\cA_{Z/B}}^L\cO_Z.
\end{equation}
These are quasi-isomorphisms of graded $\pi^*(\cA_B)$-complexes.

Similarly, as is well-known,
\[
\bigoplus_p{\rm Gr}_p^F{\rm DR}_{Z}(\cM^\bullet)\simeq\cN^\bullet\otimes_{\cA_Z}^L\cO_Z.\]
Consequently, transitivity of derived tensor products gives graded quasi-isomorphisms
\begin{align}
\label{eq-graded-Spencer-transitivity}
\bigoplus_p{\rm Gr}_p^F{\rm DR}_{Z}(\cM^\bullet)&\simeq\left(\cN^\bullet\otimes_{\cA_Z}^L\pi^*(\cA_B)
\right)\otimes_{\pi^*(\cA_B)}^L\cO_Z\\
&\simeq
\left(\bigoplus_p{\rm Gr}_p^F{\rm DR}_{Z/B}(\cM^\bullet)\right)\otimes_{\pi^*(\cA_B)}^L\cO_Z.
\end{align}

Indeed, the first derived tensor product is computed by the Koszul
resolution of $\pi^*(\cA_B)$ over $\cA_Z$, whose term in degree
$k$ is
\[\cA_Z(-k)\otimes_{\cO_Z}\bigwedge^k\cT_{Z/B}.\]
The second is computed by the pulled-back Koszul resolution of
$\cO_B$ over $\cA_B$, whose term in degree $j$ is
\[\pi^*(\cA_B)(-j)\otimes_{\cO_Z} \pi^*\left(\bigwedge^j\cT_B\right).\]

\subsection{Smooth morphisms}
We can now verify the formula from the introduction, showing that our notion of relative Du Bois complexes gives the right answer for a smooth morphism.

\begin{eg}[Smooth Morphisms] \label{eg-SmoothMaps} Let $f\colon X \to B$ be a smooth morphism with $X$ and $B$ both smooth. Then the short exact sequence
\[ 0 \to f^* \Omega_B^1 \to \Omega_X^1 \to \Omega_{X/B}^1 \to 0\]
gives a resolution for any $p\geq 0$:
\[ 0 \to {\rm Sym}^p(f^*(\Omega_B^1)) \to {\rm Sym}^{p-1}(f^*(\Omega_B^1)) \otimes \Omega_X^1 \to \dots \to f^* \Omega_B^1 \otimes \Omega_X^{p-1} \to \Omega_X^p\]
of $\Omega_{X/B}^p = \bigwedge^p \Omega_{X/B}^1$. We write $\cC^\bullet_p$ for the complex, which is a perfect complex that is quasi-isomorphic to $\Omega_{X/B}^p$.

Then the Grothendieck dual of $\Omega_{X/B}^p$ is given by
\[ R \cH om_{\cO_X}(\Omega_{X/B}^p,\omega_X[\dim X]) =  \cH om_{\cO_X}(\cC^\bullet_p,\omega_X[\dim X]),\]
and hence can be represented by the complex
\begin{equation} \label{eq-DualRelKahler} \omega_X \otimes_{\cO_X} \bigwedge^p \cT_X \to \omega_X \otimes f^*(\cT_B) \otimes \bigwedge^{p-1} \cT_X \to \dots \to \omega_X \otimes {\rm Sym}^p f^*(\cT_B)\end{equation}
placed in degrees $-\dim X,\dots, -\dim X + p$.

We can factor $f \colon X \to B$ through the graph embedding, yielding $X \xrightarrow[]{\iota} B\times X \xrightarrow[]{p} B$. Then, because $X$ is smooth, we have the Poincar\'{e} duality isomorphism 
\[ \mathbf D({\rm IC}_X^H[-\dim X]) \cong \mathbf D_X(\Q_X^H) \cong (\Q_X^H[\dim X](\dim X))[\dim X].\] 

The (complex of) right $\cD_{B\times X}$-module(s) underlying $\iota_* \mathbf D_X(\Q_X^H)$ is $\iota_+(\omega_X)[\dim X]$. If $b_1,\dots, b_{d_B}$ are coordinates on $B$, then we can write
\[ \iota_+(\omega_X) = \bigoplus_{\alpha \in \N^{d_B}} \omega_X \de_b^\alpha \delta_f,\]
where the action of $\theta \in \cT_X$ is given by
\[ (\omega \de_b^\alpha \delta_f)\theta = (\omega \theta) \de_b^\alpha\delta_f - \sum_{i=1}^{d_B} \theta(f^*(b_i)) \omega \de_b^{\alpha+e_i} \delta_f\]
and the Hodge filtration is (keeping in mind the Tate twist):
\[ F_{p} \iota_+(\omega_X) = \bigoplus_{|\alpha| \leq p} \omega_X \de_b^\alpha \delta_f, \quad {\rm Gr}^F_{p} \iota_+(\omega_X) = \bigoplus_{|\alpha| = p} \omega_X \de_b^\alpha \delta_f. \]

Thus, we have
\[ {\rm Gr}^F_p {\rm DR}_{B\times X/B}(\iota_* \mathbf D_X(\Q_X^H))\]
\[= [{\rm Gr}^F_{p-d_X} \iota_+ \omega_X \otimes_{\cO} \bigwedge^{d_X} \cT_X \to {\rm Gr}^F_{p+1-d_X} \iota_+ \omega_X\otimes_{\cO} \bigwedge^{d_X-1} \cT_X \to \dots \to {\rm Gr}^F_{p} \iota_+ \omega_X],\]
placed in degrees $-2d_X, \dots, -d_X$. Using the description above, this is exactly the complex \eqref{eq-DualRelKahler} shifted to the left by $p$. Putting this together, we conclude
\[ \mathbb D_X({\rm Gr}^F_p {\rm DR}_{B\times X/B}(\iota_* \mathbf D_X(\Q_X^H)))[p] \cong \Omega_{X/B}^p,\]
as claimed.
\end{eg}

\subsection{Independence of Factorization and Restriction of the Base}
In this subsection, we check that the relative filtered de Rham complex is independent of the chosen factorization $\iota \colon X \hookrightarrow B \times Y$. Note that, for $M^\bullet \in D^b({\rm MHM}(X))$ with $(\cM^\bullet,F)$ underlying $\iota_* M^\bullet \in D^b({\rm MHM}(B \times Y))$, we have
\begin{equation} \label{eq-SuppX} {\rm Gr}^F_p {\rm DR}_{B\times Y}(\cM^\bullet), {\rm Gr}^F_p {\rm DR}_{B\times Y/B}(\cM^\bullet)  \in D^b_{\rm coh}(\cO_X).\end{equation}

Indeed, the filtered $\cD$-modules underlying mixed Hodge modules with support in $X$ satisfy
\[ I \cdot {\rm Gr}^F_p  \cM^j = 0 \text{ for all } p \text{ and } j \in \Z,\]
where $I \subseteq \cO_{B\times Y}$ is the ideal sheaf defining $X$. 

Given $\iota_a \colon X \hookrightarrow B \times Y_a$ for $a\in\{1,2\}$, consider the diagonal embedding
\[ \iota\colon X \hookrightarrow B \times Y_1 \times Y_2.\]

Let $q_a \colon B \times Y_1 \times Y_2 \to B \times Y_a$ be the projection. By construction, we have
\begin{equation} \label{eq-compareEmbeddings} \iota_a = q_a\circ \iota \colon X \to B \times Y_1 \times Y_2 \to B \times Y_a.\end{equation}

The primary independence check is the following:
\begin{lem} \label{lem-independent} In the above notation, there is a natural quasi-isomorphism
\[ R q_{a,*}^{\cO} {\rm Gr}^F_p {\rm DR}_{B\times Y_1\times Y_2/B}(\iota_* \cM^\bullet) \simeq {\rm Gr}^F_p {\rm DR}_{B\times Y_a/B}(\iota_{a,*} \cM^\bullet)\]
for any $a\in \{1,2\}$ and $p\in \Z$.
\end{lem}
\begin{proof} The relation \eqref{eq-compareEmbeddings} gives a natural filtered isomorphism
\[ \iota_{a,*}(\cM^\bullet,F) \simeq q_{a,+}( \iota_*(\cM^\bullet,F)).\]

By definition, the right-hand side is defined as
\[ R q_{a,*}^{\cO}( {\rm DR}_{B \times Y_1 \times Y_2/ B\times Y_a}(\iota_*(\cM^\bullet,F))).\]

If we then apply ${\rm Gr}^F_p {\rm DR}_{B \times Y_a/B}(-)$ to both sides of this filtered quasi-isomorphism, we get the desired quasi-isomorphism.
\end{proof}

For three factorizations $\iota_a \colon X \hookrightarrow B \times Y_a$, a similar triple-product argument shows that the comparison quasi-isomorphisms are compatible, in the obvious way.

\begin{defi} Let $f\colon X \to B$ be smooth-factorizable and let $M^\bullet \in D^b({\rm MHM}(X))$. We define for any $p \in \Z$ the \emph{$p$th relative de Rham complex}
\[ {\rm Gr}^F_p {\rm DR}_{X/B}(M^\bullet) = {\rm Gr}^F_p {\rm DR}_{B \times Y/B}(\cM^\bullet),\]
for any smooth factorization $\iota\colon X \hookrightarrow B \times Y$ of $f$. Here $(\cM^\bullet,F)$ is a bounded complex of filtered $\cD_{B\times Y}$-modules underlying $\iota_* M^\bullet$.
\end{defi}

We also have the comparison for open subsets on the base:
\begin{lem} Let $f \colon X \to B$ be a smooth-factorizable morphism. Let $V\subseteq B$ be a non-empty open subset with inverse image $j\colon f^{-1}(V) \to X$. Then for all $p\in \Z$ and $M^\bullet \in D^b({\rm MHM}(X))$, the natural map
\[ j^* {\rm Gr}^F_p {\rm DR}_{X/B}(M^\bullet) \to {\rm Gr}^F_p {\rm DR}_{f^{-1}(V)/V}(j^*(M^\bullet))\]
is a quasi-isomorphism.  
\end{lem}
\begin{proof} A smooth factorization of $f$ induces one for $f\vert_{f^{-1}(V)}$, and so the claim is clear by restricting the relative Spencer complex.
\end{proof}

\begin{rmk} When $f\colon X \to B = {\rm Spec}(\C)$ is the constant map, smooth-factorizability of $f$ is equivalent to $X$ being smooth-embeddable. In this case,
\[ {\rm Gr}^F_p {\rm DR}_{X/B}(M^\bullet) = {\rm Gr}^F_p {\rm DR}_{X}(M^\bullet)\]
is the usual associated graded piece of the de Rham complex.
\end{rmk}

We give the definition of the relative (intersection) Du Bois complexes from the introduction.

\begin{defi} Let $f\colon X \to B$ be a smooth-factorizable morphism with $B$ smooth and connected. 

We define
\[ \underline{\Omega}_{X/B}^p \coloneqq \mathbb D_X( {\rm Gr}^F_p {\rm DR}_{X/B}(\mathbf D_X(\Q_X^H)))[p].\]

Similarly, we define
\[ {\rm I}\underline{\Omega}_{X/B}^p \coloneqq \mathbb D_X( {\rm Gr}^F_p {\rm DR}_{X/B}(\mathbf D_X({\rm IC}_X^H)))[p-\dim X].\]

The morphism $\Q_X^H[\dim X] \to {\rm IC}_X^H$ induces a map
\[ \mathbf D({\rm IC}_X^H)[\dim X] \to \mathbf D(\Q_X^H)\]
and thus, by functoriality of ${\rm Gr}^F_p{\rm DR}_{X/B}(-)$, a morphism
\[ \underline{\Omega}_{X/B}^p \to {\rm I}\underline{\Omega}_{X/B}^p.\]
\end{defi} 

\begin{rmk} \label{rmk-LowestPiece} The construction leads to the vanishing
\[ \underline{\Omega}_{X/B}^p = {\rm I}\underline{\Omega}_{X/B}^p = 0 \text{ for all }p < 0.\]

Indeed, for any closed embedding $\iota \colon X \hookrightarrow Z$ with $Z$ smooth, let $(\cM^\bullet,F)$ underlie $\iota_* \mathbf D(\Q_X^H)$ and let $(\cL,F)$ underlie $\iota_* \mathbf D({\rm IC}_X^H)$. Then
\begin{equation} \label{eq-HodgeVanishing} F_{-1} \cM^\bullet \simeq 0 = F_{-1}\cL.\end{equation}

Indeed,
\[ \min \{p \mid F_p \cM^\bullet \neq 0\} = \min\{p \mid {\rm Gr}^F_p {\rm DR}(\cM^\bullet)\neq 0\}\]
and the same holds for $(\cL,F)$. Thus, the claimed vanishing follows from the quasi-isomorphisms
\[ {\rm Gr}^F_p {\rm DR}(\cM^\bullet) \cong {\rm Gr}^F_p {\rm DR}_X(\mathbf D(\Q_X^H)) \cong \mathbb D_X {\rm Gr}^F_{-p} {\rm DR}_X(\Q_X^H) \cong \mathbb D_X(\underline{\Omega}_X^p[-p]),\]
\[ {\rm Gr}^F_p {\rm DR}_Z(\cL) \cong {\rm Gr}^F_p {\rm DR}_X(\mathbf D_X({\rm IC}_X^H)) \cong \mathbb D_X{\rm Gr}^F_{-p} {\rm DR}_X({\rm IC}_X^H) \cong \mathbb D_X({\rm I}\underline{\Omega}_X^p[\dim X-p]),\]
and the right-hand sides vanish for $p < 0$.

If we apply this with $Z = B \times Y$ for some smooth factorization of $f$, then the construction of the relative Spencer complex implies
\[ {\rm Gr}^F_p {\rm DR}_{B\times Y/B}(\cM^\bullet) = {\rm Gr}^F_p {\rm DR}_{B\times Y/B}(\cL) = 0 \text{ for all } p < 0,\]
as claimed.
\end{rmk}

\section{The absolute-relative filtration} \label{sec-filtration}
In this section, we combine the duality with a standard filtration result for a double complex to filter the Du Bois complex by twisted relative Du Bois pieces. As before, let $Z = B \times Y$ for $B,Y$ smooth varieties and let $\pi \colon Z \to B$ be the first projection. Note there is a natural isomorphism
\[ \cD_Z \cong \cD_B \boxtimes \cD_Y\]
and
\[ \bigwedge^\ell \cT_Z = \bigoplus_{j =0}^\ell \left(\bigwedge^j \cT_B \boxtimes \bigwedge^{\ell-j} \cT_Y\right)= \bigoplus_{j=0}^\ell \left(\pi^*(\bigwedge^j \cT_B) \otimes_{\cO_Z} \bigwedge^{\ell-j} \cT_{Z/B}\right).\]

\begin{lem} \label{lem-ColumnFiltration} Let $(\cM^\bullet,F)$ be a bounded complex of filtered right $\cD_Z$-modules. Then, for any $p\in \Z$, there are morphisms of bounded complexes of $\cO_Z$-modules
\[ H_{-1} = 0 \to H_0 {\rm Gr}^F_p {\rm DR}_Z(\cM^\bullet) \to H_1 {\rm Gr}^F_p {\rm DR}_Z(\cM^\bullet) \to \ldots \to H_{\dim B} {\rm Gr}^F_p {\rm DR}_Z(\cM^\bullet) = {\rm Gr}^F_p {\rm DR}_Z(\cM^\bullet) \]
such that, writing $H_r \coloneqq H_r {\rm Gr}^F_p {\rm DR}_Z(\cM^\bullet)$, we have a quasi-isomorphism for all $0 \leq j \leq \dim B$:
\begin{equation} \label{eq-GrHCompute} {\rm cone}(H_{j-1} \to H_j) \simeq {\rm Gr}^F_{p-j} {\rm DR}_{Z/B}(\cM^\bullet) \otimes \pi^*(\bigwedge^j \cT_B)[j].\end{equation}

Given any morphism $\varphi \colon (\cM_1^\bullet,F) \to (\cM_2^\bullet,F)$, the morphism induced by applying ${\rm Gr}^F_p {\rm DR}_Z(-)$ is $H$-filtered in the obvious sense.
\end{lem}
\begin{proof}
Define
\[
\cC^\bullet_{Z/B}
\coloneqq 
\bigoplus_a{\rm Gr}_a^F{\rm DR}_{Z/B}(\cM^\bullet).
\]
By \eqref{eq-graded-Spencer-transitivity},
\[
\bigoplus_p{\rm Gr}_p^F{\rm DR}_Z(\cM^\bullet)
\simeq
\cC^\bullet_{Z/B}
\otimes_{\pi^*(\cA_B)}^L\cO_Z.
\]
We can compute the right-hand side using the graded Koszul resolution of
$\cO_Z$ over $\pi^*(\cA_B)$, and filter the resulting resolution by the degree of the wedge powers. If $H_r$ denotes the subcomplex consisting of wedge-power degrees $0\leq j\leq r$, then
\[H_r/H_{r-1} \simeq\cC^\bullet_{Z/B}(-r)\otimes_{\cO_Z} \pi^*\left(\bigwedge^r\cT_B\right)[r].\]
Taking degree $p$ gives \[{\rm Gr}_r^H{\rm Gr}_p^F{\rm DR}_Z(\cM^\bullet)
\simeq {\rm Gr}_{p-r}^F{\rm DR}_{Z/B}(\cM^\bullet)
\otimes_{\cO_Z} \pi^*\left(\bigwedge^r\cT_B\right)[r].\]
\end{proof}

\begin{rmk} Using the diagonal embedding argument as in \lemmaref{lem-independent}, we can see that the $G$-filtration is independent of the factorization of the map $f\colon X \to B$.
\end{rmk}

Then \theoremref{thm-filtration} is a consequence of the definition in terms of duality:
\begin{proof}[Proof of \theoremref{thm-filtration}]
As $f\colon X \to B$ is smooth-factorizable, choose a factorization $f = \pi \circ \iota \colon X \to Z \to B$ through $Z = B \times Y$ with $Y$ smooth.

Consider the morphism $\Q_X^H[\dim X] \to {\rm IC}_X^H$ in $D^b({\rm MHM}(X))$, which we pushforward and dualize to get a morphism
\[ \iota_* \mathbf D({\rm IC}_X^H) \to \iota_* \mathbf D(\Q_X^H[\dim X]) \text{ in } D^b({\rm MHM}(Z)).\]

In terms of underlying filtered $\cD_Z$-modules, we can represent this as a morphism
\[ (\cM_1^\bullet,F) \xrightarrow[]{\varphi} (\cM_2^\bullet,F).\]

For $p\in \Z$ fixed, if we apply ${\rm Gr}^F_p {\rm DR}_Z(-)$ to this morphism, we get a map
\[ {\rm Gr}^F_p {\rm DR}(\cM_1^\bullet) \to {\rm Gr}^F_p {\rm DR}(\cM_2^\bullet),\]
where both the source and target are $H$-filtered, and this morphism preserves the filtration in the sense that it induces morphisms
\[ H_{j} {\rm Gr}^F_p {\rm DR}(\cM_1^\bullet) \to H_{j} {\rm Gr}^F_p {\rm DR}(\cM_2^\bullet),\]
for all $j\in \{0,\ldots, \dim B\}$. 

For $i\in \{1,2\}$, $p\in \Z$, and $j\in \{0,\ldots,\dim B\}$, we consider the quotient map
\[ {\rm Gr}^F_p {\rm DR}(\cM_i^\bullet) \to {\rm Gr}^F_p {\rm DR}(\cM_i^\bullet)/ H_{j}.\]

There are exact triangles for any $0 \leq j \leq \dim B$:
\begin{equation} \label{eq-GrQuotients} {\rm Gr}^H_{j} {\rm Gr}^F_p {\rm DR}_Z(\cM_i^\bullet) \to {\rm Gr}^F_p {\rm DR}(\cM_i^\bullet)/H_{j-1} \to {\rm Gr}^F_p {\rm DR}(\cM_i^\bullet)/H_j \xrightarrow[]{+1} \end{equation}

Define
\[ G_{-j} \mathbb D_X({\rm Gr}^F_p {\rm DR}_Z(\cM_i^\bullet)) \coloneqq \mathbb D_X( {\rm Gr}^F_p{\rm DR}_Z(\cM_i^\bullet)/H_{j-1}),\]
so that we have exact triangles
\begin{equation} \label{eq-GrDual} G_{-j-1}\mathbb D_X({\rm Gr}^F_p {\rm DR}_Z(\cM_i^\bullet))   \to G_{-j}\mathbb D_X({\rm Gr}^F_p {\rm DR}_Z(\cM_i^\bullet))   \to \mathbb D_X( {\rm Gr}^H_{j} {\rm Gr}^F_p {\rm DR}_Z(\cM_i^\bullet)) \xrightarrow[]{+1}.\end{equation}

By construction, we have morphisms
\[ G_{-j} \mathbb D_X( {\rm Gr}^F_p {\rm DR}(\cM_i^\bullet)) \to \mathbb D_X({\rm Gr}^F_p {\rm DR}(\cM_i^\bullet))\]
which form commutative diagrams with the corresponding morphisms induced by $\varphi$.

The proof is complete once we observe that
\[ \mathbb D_X({\rm Gr}^F_{p} {\rm DR}(\cM_1^\bullet)) \simeq {\rm Gr}^F_{-p} {\rm DR}({\rm IC}_X^H) = {\rm I}\underline{\Omega}_X^p[\dim X-p], \]
\[ \mathbb D_X({\rm Gr}^F_{p} {\rm DR}(\cM_2^\bullet)) \simeq {\rm Gr}^F_{-p} {\rm DR}(\Q_X^H[\dim X]) = \underline{\Omega}_X^p[\dim X-p],\]
the induced morphism is the canonical morphism between the two, and we have the isomorphisms (by combining \eqref{eq-GrHCompute} and \eqref{eq-GrDual}):
\[ {\rm Gr}^G_{-j} {\rm I}\underline{\Omega}_X^p \simeq \mathbb D_X(  {\rm Gr}^H_{j} {\rm Gr}^F_p {\rm DR}(\cM_1^\bullet))[p-\dim X] \simeq \mathbb D_X( {\rm Gr}^F_{p-j} {\rm DR}_{Z/B}(\cM_1^\bullet) \otimes \bigwedge^j \pi^*(\cT_B)[j])[p-\dim X], \]
\[ {\rm Gr}^G_{-j} \underline{\Omega}_X^p \simeq \mathbb D_X( {\rm Gr}^H_{j} {\rm Gr}^F_p {\rm DR}(\cM_2^\bullet))[p-\dim X] \simeq \mathbb D_X( {\rm Gr}^F_{p-j} {\rm DR}_{Z/B}(\cM_2^\bullet) \otimes \bigwedge^j \pi^*(\cT_B)[j])[p-\dim X].\]

Using the fact that $\pi^*(\cT_B)$ is a vector bundle on $Z$, we can rewrite the right-hand sides of these isomorphisms as in the theorem statement.
\end{proof}

\section{Restriction to Fibers} \label{sec-BC} In this section, we construct a base-change morphism via filtered $\cD$-module theory. We then prove generic base change \theoremref{thm-generalBC} and finally discuss the base change results of \cite{CDOHigher}*{Sec. 9} and their connection with our relative Du Bois complexes.

\subsection{Filtered $\cD$-module Restriction Maps}
We recall the base-change morphism for filtered $\cD$-modules underlying mixed Hodge modules \cite{SchnellNeron}*{Lem. 2.17} (see also \cite{SchnellWeak}*{Lem. 3.2}). The morphism is essentially induced by the counit map
\[ i_* i^! \cM^\bullet \to \cM^\bullet.\]

Our statement differs from Schnell's because we use right $\cD$-modules. This explains why we do not shift the filtration on the domain of the morphism and why we use $Li^!_{\cO}(-)$ (for $\cO$-modules) in place of $Li^*(-)$.

\begin{lem} Let $i \colon Y \hookrightarrow Z$ be a closed embedding of smooth varieties. Let $(\cM^\bullet,F)$ be a bounded complex of filtered $\cD_Z$-modules underlying $M^\bullet \in D^b({\rm MHM}(Z))$. Then for all $p \in \Z$, there are canonical morphisms
\[ \eta_{Y,Z,p}^{\cM^\bullet} \colon F_p i^!(\cM^\bullet) \to Li^!( F_p \cM^\bullet),\]
\[ \chi_{Y,Z,p}^{\cM^\bullet}  \colon {\rm Gr}^F_p i^!(\cM^\bullet) \to Li^!( {\rm Gr}^F_p \cM^\bullet)\]
inducing a morphism of exact triangles
\[ \begin{tikzcd} F_{p-1} i^!(\cM^\bullet) \ar[r] \ar[d,"\eta_{p-1}"] & F_p i^!(\cM^\bullet) \ar[r] \ar[d,"\eta_p"] & {\rm Gr}^F_p i^!(\cM^\bullet) \ar[d,"\chi_p"] \ar[r,"+1"] & {} \\ Li^!( F_{p-1} \cM^\bullet) \ar[r] &  Li^!( F_p \cM^\bullet) \ar[r] &  Li^!( {\rm Gr}^F_p \cM^\bullet) \ar[r,"+1"] & {}\end{tikzcd}. \]

The morphisms are compatible with composition in the sense that, if $i = i_2 \circ i_1 \colon Y \hookrightarrow  \widetilde{Z} \hookrightarrow Z$ is a factorization of the map $i$ through another smooth closed subvariety $\widetilde{Z}\subseteq Z$, then
\[ \eta_{Y,Z,p}^{\cM^\bullet}  = L i_1^!(\eta_{\widetilde{Z},Z,p}^{\cM^\bullet})\circ \eta_{Y,\widetilde{Z},p}^{i_2^! \cM^\bullet},\]
\[ \chi_{Y,Z,p}^{\cM^\bullet}  = L i_1^!(\chi_{\widetilde{Z},Z,p}^{\cM^\bullet})\circ \chi_{Y,\widetilde{Z},p}^{i_2^! \cM^\bullet}.\]
\end{lem}

The last statement follows from the compatibility of counits for the adjunction.

\subsection{Relative de Rham Restriction}
We now define a base-change morphism for (relative) de Rham complexes. 

Let $Z = B \times Y$, with $B$ and $Y$ both smooth, and let $k_b\colon Z_b = \{b\} \times Y \hookrightarrow Z$ denote the fiber inclusion. 

\begin{lem} \label{lem-SmoothFactBC} Let $(\cM^\bullet,F)$ be a bounded complex of filtered right $\cD_Z$-modules underlying a bounded complex of mixed Hodge modules on $Z$. Then for every $p$, there is a natural base-change morphism
\begin{equation} \label{eq-BCMap} \gamma_{b,Z,p}^{\cM^\bullet} \colon {\rm Gr}^F_p {\rm DR}_{Z_b}(k_b^!(\cM^\bullet)) \to Lk_b^! {\rm Gr}^F_p {\rm DR}_{Z/B}(\cM^\bullet).\end{equation}
\end{lem}
\begin{proof} By taking the direct sum over $p$ of the maps $\chi_{Z_b,Z,p}^{\cM^\bullet}$, we get a graded morphism
\[ \bigoplus_{p\in \Z} {\rm Gr}^F_p k_b^!(\cM^\bullet) \to Lk_b^!(\bigoplus_{p\in \Z} {\rm Gr}^F_p \cM^\bullet),\]
which is linear over the graded algebra $\cA_{Z_b} = k_b^* \cA_{Z/B}$. The desired morphism is obtained by derived tensoring over $\cA_{Z_b}$ with $\cO_{Z_b}$, using the graded quasi-isomorphism relation
\[ (L k_b^! \cN^\bullet) \otimes_{\cA_{Z_b}}^L \cO_{Z_b} \simeq Lk_b^! ( \cN^\bullet \otimes^L_{\cA_{Z/B}} \cO_Z),\]
for any graded complex $\cN^\bullet$ over $\cA_{Z/B}$.
\end{proof}

\subsection{Du Bois base-change morphism}
Applying the above construction in a special case gives the Du Bois fiber-restriction morphism for flat, locally smooth-factorizable morphisms.

\begin{prop} For any flat, locally smooth-factorizable morphism $f\colon X \to B$  with $B$ smooth and connected, and any $b \in B$, there are base-change morphisms
\[ Li_b^*(\underline{\Omega}_{f^{-1}(U)/U}^p) \to \underline{\Omega}_{X_b}^p,\]
for all $p\in \Z$, where $b\in U \subseteq B$ is an open neighborhood such that $f^{-1}(U) \to U$ is smooth-factorizable.
\end{prop}
\begin{proof} Fix $b\in B$ and replace $B$ with $U$ and $X$ with $f^{-1}(U)$ to assume from the beginning that $f$ is smooth-factorizable. Let $X \xrightarrow[]{\iota} Z = B \times Y \xrightarrow[]{\pi} B$ be a smooth factorization of $f$. Consider the Cartesian diagram
\[ \begin{tikzcd} X_b \ar[r,"i_b"] \ar[d,"\iota_b"] & X \ar[d,"\iota"] \\ Z_b = \{b\} \times Y \ar[r,"k_b"] & Z \end{tikzcd}.\]

Let $(\cM^\bullet,F)$ be the bounded complex of filtered $\cD_{Z}$-modules underlying $\iota_*\mathbf D(\Q_X^H)$. Then \lemmaref{lem-SmoothFactBC} gives the canonical base-change morphism
\[ {\rm Gr}^F_p {\rm DR}_{Z_b}(k_b^!(\cM^\bullet)) \to Lk_b^! {\rm Gr}^F_p {\rm DR}_{Z/B}(\cM^\bullet)\]
of complexes supported on $X_b$. 

We have a canonical quasi-isomorphism 
\[ k_b^! \iota_* \mathbf D(\Q_X^H) \simeq \iota_{b,*} \mathbf D(\Q_{X_b}^H),\]
and so, since $X_b\hookrightarrow X$ is a regular embedding (by flatness of $f$ and smoothness of $B$), we can rewrite this morphism as
\[ \iota_{b,*} {\rm Gr}^F_p {\rm DR}_{X_b}(\mathbf D_{X_b}(\Q_{X_b}^H)) \to \iota_{b,*} Li_b^! {\rm Gr}^F_p {\rm DR}_{X/B}(\mathbf D_X(\Q_X^H)).\]
 
If we apply $\mathbb D_{X_b}(-)$ to this morphism, we get a morphism
\[ Li_b^* \mathbb D_X{\rm Gr}^F_p {\rm DR}_{X/B}(\mathbf D_X(\Q_X^H)) \to \mathbb D_{X_b} {\rm Gr}^F_p {\rm DR}(\mathbf D(\Q_{X_b}^H)) \simeq {\rm Gr}^F_{-p} {\rm DR}_{X_b}(\Q_{X_b}^H),\]
where the isomorphism on the right holds because the fiber $X_b$ is a closed subvariety of the smooth variety $Y$. Up to shift, this gives the base-change morphism in the proposition statement.

By comparing two embeddings to the induced diagonal embedding over $B$, one can check that this is independent of the choice of smooth factorization.
\end{proof}

\subsection{Quasi-isomorphism criterion}
We give a $\cD$-module theoretic criterion for the base-change morphism to be a quasi-isomorphism in some range.

In the codimension $1$ case, the inverse image $i^!(-)$ for mixed Hodge modules can be defined by the Kashiwara--Malgrange $V$-filtration \cite{SaitoMHM}*{Cor. 2.24}. For details on this filtration and its behavior for filtered $\cD$-modules underlying Hodge modules, see \cite{SaitoMHP}*{Sec. 3}.

One can make the base-change morphism for filtered $\cD$-modules precise via the $V$-filtration. By the compatibility with composition, using local coordinates, we can always reduce to this case for computations. 

We have the following inductive criterion for the base-change map to be a quasi-isomorphism for some range of $p$:
\begin{prop} \label{prop-VContainBC} Let $i \colon Y \hookrightarrow Z$ be a closed embedding of smooth varieties and let $(\cM^\bullet,F)$ be a complex of filtered right $\cD_Z$-modules underlying $M^\bullet \in D^b({\rm MHM}(Z))$. 

Fix an integer $m\in \Z$. Let $Z = \bigcup_{\alpha \in \Lambda} U_\alpha$ be an open cover such that, for every $\alpha\in \Lambda$ with $Y \cap U_\alpha \neq \emptyset$, there are $t_1,\dots, t_r \in \cO_{U_\alpha}(U_\alpha)$ forming a partial set of coordinate functions defining $Y$ inside $U_\alpha$, so that if we set
\[ Y \cap U_\alpha = H_r \subseteq H_{r-1} = V(t_1,\dots, t_{r-1}) \subseteq \dots \subseteq H_1 = V(t_1) \subseteq H_0 = U_\alpha\]
with closed embeddings $\kappa_j \colon H_j \hookrightarrow H_{j-1}$, then we have containment for all $\ell \in \Z$ and all $j\in \{1,\dots, r\}$:
\[ F_m \cH^\ell \kappa_{j-1}^! \ldots \kappa_1^! \cM^\bullet \subseteq V^1_{t_j} \cH^\ell\kappa_{j-1}^! \ldots \kappa_1^! \cM^\bullet.\]

Then the base-change morphisms
\[\eta_{Y,Z,p}^{\cM^\bullet} \text{ and } \chi_{Y,Z,p}^{\cM^\bullet}\]
are quasi-isomorphisms for all $p \leq m$.
\end{prop}
\begin{proof} First of all, since the base change morphisms are globally defined, we can replace $Z$ with $U_\alpha$ and $Y$ with $U_\alpha \cap Y$ to assume that we have a partial system of local coordinates $t_1,\dots, t_r\in \cO_Z(Z)$ defining $Y$. Moreover, by the compatibility with composition, it suffices to assume $r=1$, and we write $t= t_1$ in this case.

Thus, we have reduced to checking: if $(\cM^\bullet,F)$ satisfies 
\[ F_m \cH^\ell \cM^\bullet \subseteq V^1_t \cH^\ell \cM^\bullet\]
for all $\ell \in \Z$, then the base-change morphisms are quasi-isomorphisms for all $p\leq m$. Note that if $\eta_{Y,Z,p}^{\cM^\bullet}$ and $\eta_{Y,Z,p-1}^{\cM^\bullet}$ are quasi-isomorphisms, then so is $\chi_{Y,Z,p}^{\cM^\bullet}$, and so it suffices to prove the claim for the $\eta$ morphisms.

By induction on the cohomological amplitude of the complex $(\cM^\bullet,F)$, we can reduce via naturality of the maps involved to the case where $(\cM^\bullet,F) = (\cM,F)$ is a single filtered $\cD_Z$-module underlying a mixed Hodge module on $Z$, such that
\[ F_m \cM \subseteq V^1_t \cM.\]

Now the description of the base-change map in terms of the $V$-filtration comes in. By definition \cite{SaitoMHM}*{Cor. 2.24}, $F_p i^!(\cM)$ is given by the morphism
\[ F_p {\rm Gr}_V^0(\cM) = \frac{F_p \cM \cap V^0 \cM}{F_p \cM \cap V^{>0} \cM} \xrightarrow[]{t} \frac{F_p \cM \cap V^1 \cM}{F_p \cM\cap V^{>1}\cM} = F_p {\rm Gr}_V^1(\cM)\]
placed in degrees $0,1$. But for filtered $\cD_Z$-modules underlying mixed Hodge modules, we have the isomorphism \cite{SaitoMHP}*{(3.2.1.2)} 
\[ t\colon F_p\cM \cap V^{>0} \cM \cong F_p \cM \cap V^{>1} \cM,\]
and thus, in the following roof:
\[ [F_p {\rm Gr}_V^0(\cM) \xrightarrow[]{t} F_p {\rm Gr}_V^1(\cM)] \leftarrow [F_pV^0\cM \xrightarrow[]{t} F_pV^1\cM ] \to [F_p \cM \xrightarrow[]{t} F_p \cM ] \]
the left morphism is a quasi-isomorphism. This gives the morphism $\eta_{Y,Z,p}^{\cM}$.

Finally, if $F_m \cM \subseteq V^1_t \cM$, then $F_p \cM\subseteq V^1_t \cM \subseteq V^0_t \cM$ for all $p \leq m$, too. Thus, the right morphism in the roof is an equality, and so we conclude the desired quasi-isomorphism.
\end{proof}

\subsection{Generic Base Change}
We recall the non-characteristic property for filtered $\cD$-modules. The point will be that, over an open subset in the base, the fiber inclusions are non-characteristic for any finite collection of filtered $\cD$-modules underlying mixed Hodge modules.

\begin{defi} Let $Z$ be a smooth complex variety and let $(\cM,F)$ be a filtered $\cD_Z$-module. If $Y$ is a smooth closed subvariety with closed embedding $i\colon Y \to Z$, we say $i$ (or $Y$) is \emph{non-characteristic} with respect to $(\cM,F)$ if
\begin{enumerate} \item \label{itm-OZFlat} For any $p\in \Z$, we have
\[ {\rm Gr}^F_p \cM \otimes_{\cO_Z}^L \cO_Y \simeq {\rm Gr}^F_p \cM \otimes_{\cO_Z} \cO_Y.\]

\item \label{itm-nonCharD} The map $i$ is non-characteristic with respect to the underlying $\cD_Z$-module $\cM$, as defined in \cite{HTT}*{Def. 2.4.2}.
\end{enumerate}
\end{defi}

The second condition holds if there is a Whitney stratification $Z = \sqcup_{\alpha \in I} S_\alpha$ adapted to the $\cD_Z$-module $\cM$, such that $Y$ intersects each $S_\alpha$ transversally.

The non-characteristic property controls the restriction of the filtered $\cD$-module to $Y$, as in the following corollary:
\begin{cor} \label{cor-BCNonChar} Let $i \colon Y \hookrightarrow Z$ be a closed embedding of smooth varieties and let $(\cM^\bullet,F)$ be a complex of filtered right $\cD_Z$-modules underlying $M^\bullet \in D^b({\rm MHM}(Z))$. 

If $i$ is non-characteristic for $\cH^\ell(\cM^\bullet,F)$ for all $\ell \in \Z$, then the base-change morphisms
\[\eta_{Y,Z,p}^{\cM^\bullet} \text{ and } \chi_{Y,Z,p}^{\cM^\bullet}\]
are quasi-isomorphisms for all $p \in \Z$.
\end{cor}
\begin{proof} Locally, we can assume $Y$ is defined inside $Z$ by some partial system of coordinates $t_1,\dots, t_r$ as in \propositionref{prop-VContainBC}. Factoring the inclusion $i \colon Y \hookrightarrow Z$ as
\[ Y = H_r \subseteq H_{r-1} \subseteq \dots \subseteq H_0= Z\]
note that \cite{SaitoMHP}*{Lem. 3.5.4(1)} implies that the closed embedding $\kappa_{j} \colon H_{j} \to H_{j-1}$ is non-characteristic with respect to $\cH^\ell \kappa_{j-1}^!\ldots \kappa_1^! (\cM^\bullet,F)$ for every $\ell \in \Z$.

Then, \cite{SaitoMHP}*{Lem. 3.5.6} shows that
\[ V^1_{t_j} \cH^\ell \kappa_{j-1}^!\ldots \kappa_1^! (\cM^\bullet) = \cH^\ell \kappa_{j-1}^!\ldots \kappa_1^! (\cM^\bullet),\]
and so in particular, 
\[ F_m \cH^\ell \kappa_{j-1}^!\ldots \kappa_1^! (\cM^\bullet) \subseteq V^1_{t_j} \cH^\ell \kappa_{j-1}^!\ldots \kappa_1^! (\cM^\bullet)\]
for all $\ell, m \in \Z$. Thus, the claim follows from \propositionref{prop-VContainBC}.
\end{proof}

We can now prove \theoremref{thm-generalBC}.

\begin{proof}[Proof of \theoremref{thm-generalBC}] By replacing $B$ with an appropriate open subset $U$ and $X$ with $f^{-1}(U)$, we can assume from the beginning that $f$ is flat (by generic flatness) and smooth-factorizable. 

Let $f \colon X \xrightarrow[]{\iota} B \times Y \xrightarrow[]{\pi} B$ be a smooth factorization with $Y$ smooth. Let $(\cM^\bullet,F)$ be the filtered $\cD_{B\times Y}$-module complex underlying $\iota_*\mathbf D(\Q_X^H)$. Then \cite{CDOHigher}*{Lem. 8.2} shows that there is a dense open subset $U$ of the base $B$ so that all fibers of $B\times Y\to B$ over $b\in U$ are non-characteristic with respect to the finitely many nonzero cohomology filtered $\cD$-modules $(\cH^\ell \cM^\bullet,F)$. We conclude the claim by \corollaryref{cor-BCNonChar} above.

If $(\cL,F)$ is the filtered $\cD_{B\times Y}$-module underlying $\iota_* \mathbf D({\rm IC}_X^H)$, then we can choose a dense open $U$ such that, for all $b\in U$, the fiber $\{b\} \times Y$ is non-characteristic with respect to $(\cL,F)$. 

For such a fiber, we have
\[ k_b^* \iota_* {\rm IC}_X^H \cong \iota_{b,*}{\rm IC}_{X_b}^H[\dim B]\]
(see \cite{KebekusSchnell}*{Thm. 4.16}). Dualizing, we get the relation
\[ k_b^! \iota_*\mathbf D_X({\rm IC}_X^H) \cong \iota_{b,*}\mathbf D_{X_b}({\rm IC}_{X_b}^H)[-\dim B],\]
and so the claim follows by \corollaryref{cor-BCNonChar}.
\end{proof}

\subsection{Mildly Singular Fibers}
In forthcoming joint work with Chen and Olano, we introduce three singularity notions which control Hodge theory in flat projective families. We briefly review the definitions. For more properties and examples, see \cite{CDOHigher}. For simplicity, and since it suffices for the relevant result, we assume $\iota \colon X\subseteq Y$ is a closed embedding into a smooth variety $Y$.

\begin{defi} Let $I \subseteq \cO_Y$ be the coherent radical ideal sheaf defining the reduced variety $X$ inside $Y$, and let $\cH^q_X(\omega_Y) \coloneqq \varinjlim_p \cE xt^q_{\cO_Y}(\cO_Y/I^{p+1},\omega_Y)$ denote the local cohomology modules of $\omega_Y$ along $X$. As $Y$ is smooth, we could use the equivalent description in terms of symbolic powers:
\[ \cH^q_X(\omega_Y) = \varinjlim_p \cE xt^q_{\cO_Y}(\cO_Y/I^{(p+1)},\omega_Y).\]

By Saito's theory, these filtered $\cD_Y$-modules underlie mixed Hodge modules on $Y$ (see, for example, \cite{MP3} for an in-depth discussion of this structure, or Remark \ref{rmk-Translate} below), hence admit a Hodge filtration $F_\bullet \cH^q_X(\omega_Y)$ and a weight filtration $W_\bullet \cH^q_X(\omega_Y)$. Recall that 
\begin{equation} \label{eq-WeightIC} W_{\dim Y + c} \cH^c_X(\omega_Y) \cong {\rm IC}_X^H(-c),\end{equation}
where $c \coloneqq \dim Y - \dim X$.

Let $m \in \Z_{\geq 0}$. The variety $X$ has \emph{symbolic $m$-Du Bois singularities} if, for all $q \in \Z$ and $0\leq p \leq m$, the morphism
\[ \cE xt^q_{\cO_Y}(\cO_Y/I^{(p+1)},\omega_Y) \to \cH^q_X(\omega_Y)\]
is injective with image $F_{p-\dim Y} \cH^q_X(\omega_Y)$.

The variety $X$ has \emph{strong $m$-Du Bois singularities} if, for all $q \in \Z$ and $0\leq p \leq m$, the morphism
\[ \cE xt^q_{\cO_Y}(\cO_Y/I^{p+1},\omega_Y) \to \cH^q_X(\omega_Y)\]
is injective with image $F_{p-\dim Y} \cH^q_X(\omega_Y)$.

It is shown in \cite{CDOHigher} that if $X$ is strongly $m$-Du Bois, then $I^{p+1} = I^{(p+1)}$ for all $0 \leq p\leq m$, and so it is also symbolically $m$-Du Bois.

The variety $X$ has \emph{weak $m$-rational singularities} if, for all $0 \leq p \leq m$, the image of the (injective) morphism
\[ \cE xt^c_{\cO_Y}(\cO_Y/I^{p+1},\omega_Y) \to \cH^c_X(\omega_Y)\]
is equal to $F_{p-\dim Y} W_{\dim Y +c} \cH^c_X(\omega_Y)$. Note that this definition only concerns the index $c = \dim Y - \dim X$.
\end{defi}

\begin{rmk} If $X$ has local complete intersection (LCI) singularities, then strong $m$-Du Bois is equivalent to $m$-Du Bois as defined in \cites{FL2,MPDB}.

In general, for $m=0$, strong and symbolic $0$-Du Bois are equivalent to usual Du Bois singularities, and weak $0$-rational is equivalent to weak rationality, if $X$ is assumed normal.
\end{rmk}

\begin{rmk} \label{rmk-Translate} To relate these constructions to what we have studied in this paper, note that we have the quasi-isomorphism
\[ \iota_* \iota^!( \Q_Y^H[\dim Y]) \simeq \iota_* \mathbf D_X(\Q_X^H)[-\dim Y](-\dim Y), \]
and similarly, we have
\[ \cH^q \iota_* \iota^! \Q_Y^H[\dim Y] \cong \cH^q_X(\omega_Y).\]

Thus, up to Tate twist (which just shifts Hodge and weight filtrations), the local cohomology modules in the definition of the singularity classes above are the same as the cohomology modules of $\iota_*\mathbf D_X(\Q_X^H)$, used in the definition of the relative Du Bois complex.
\end{rmk}

We will make use of the following result from \cite{CDOHigher}*{Prop. 5.6}.

\begin{prop} \label{prop-HigherSingVFilt} Let $f\colon X \to B$ be a flat projective morphism with factorization $X \xrightarrow[]{\iota} B \times \P^n \to B$. For $b_0 \in B$ fixed, after shrinking $B$ if necessary we consider a system of coordinates $t_1,\dots ,t_r$ defining $b_0$ inside $B$, with associated flag
\[ \{b_0\} \times \P^n = H_r \subseteq H_{r-1} = V(t_1,\dots, t_{r-1}) \subseteq \dots \subseteq H_0 = B \times \P^n\]
and closed embeddings $\kappa_j \colon H_j \to H_{j-1}$.

Let $(\cM^\bullet,F)$ denote the bounded complex of filtered right $\cD_{B \times \P^n}$-modules underlying $\iota_*\mathbf D(\Q_X^H)$.

If the fiber $X_{b_0} = f^{-1}(b_0)$ is symbolically $m$-Du Bois, then after possibly shrinking $B$ near $b_0$ and $X$ near $X_{b_0}$, we have containment for all $\ell \in \Z$ 
\[ F_m \cH^\ell \cM^\bullet \subseteq V^1_{t_1} \cH^\ell \cM^\bullet,\]
and for all $j\in \{2,\dots, r\}$, containment
\[ F_m \cH^\ell \kappa_{j-1}^! \ldots \kappa_1^! \cM^\bullet \subseteq V^1_{t_j} \cH^\ell\kappa_{j-1}^! \ldots \kappa_1^! \cM^\bullet.\]
\end{prop}
\begin{proof} Let $X \cap H_j = X_j$. By assumption, $X_r$ is symbolically $m$-Du Bois. Then, by Inversion of Adjunction \cite{CDOHigher}*{Thm. B} for this property, we can replace $X_{r-1}$ with an open neighborhood of the Cartier divisor $X_r\subseteq X_{r-1}$, and then assume $X_{r-1}$ is symbolically $m$-Du Bois. Repeating inductively in this way, up to replacing $X$ by an open subset, we can assume that $X_j$ is symbolically $m$-Du Bois for all $j$.

But then for a Cartier divisor $X_j \subseteq X_{j-1}$ such that both spaces are symbolically $m$-Du Bois, we have the desired $V$-filtration containment shown in \cite{CDOHigher}*{Prop. 5.6}.
\end{proof}

As mentioned above, there is, in general, no canonical comparison map between the ${\rm IC}$-objects at the filtered $\cD$-module level. In \cite{CDOHigher}*{Sec. 9}, a staircase construction is used to give a comparison, but only under mild singularity hypotheses. We summarize this in the following proposition.

\begin{prop}[\cite{CDOHigher}*{Cor. 9.5}] \label{prop-WeakMRationalCompare} Let $i \colon X' \hookrightarrow X$ be a closed embedding between pure-dimensional varieties with a flag
\[ X' = X_r \xrightarrow[]{i_r} X_{r-1} \xrightarrow[]{i_{r-1}}\dots \xrightarrow[]{i_1} X_0 = X\]
such that $X_j$ is a Cartier divisor in $X_{j-1}$ for all $j \in \{1,\dots, r\}$. Assume moreover that $X'$ is weakly $m$-rational.

The staircase argument of \cite{CDOHigher} gives, for $p\leq m$, a zigzag of quasi-isomorphisms
\[ {\rm Gr}^F_p {\rm DR}_{X'}(i^! \mathbf D({\rm IC}_X^H))[r] \simeq {\rm Gr}^F_p {\rm DR}_{X'} \cH^1 i_r^! \cH^1i_{r-1}^! \dots \cH^1 i_1^! \mathbf D({\rm IC}_X^H) \simeq {\rm Gr}^F_p {\rm DR}_{X'} \mathbf D({\rm IC}_{X'}^H).\]
\end{prop}

\begin{rmk} \label{rmk-Staircase} In the proof of the staircase comparison, \cite{CDOHigher} uses the fact that, by Inversion of Adjunction, we can replace $X_{r-1},\dots, X_0$ by open neighborhoods to assume that all of those varieties are weakly $m$-rational.

Moreover, in the smoothly embedded setting $\iota \colon X \hookrightarrow Z$, an iterated $V$-filtration containment as in the statement of \propositionref{prop-HigherSingVFilt} holds, with $(\cM,F)$ underlying $\iota_*\mathbf D({\rm IC}_X^H)$.
\end{rmk}

\begin{proof}[Proof of \theoremref{thm-higherBC}] 
The deformation result of \cite{CDOHigher}*{Cor. 8.6} says that if $X_{b_0}$ is symbolically $m$-Du Bois (respectively, weakly $m$-rational), then there exists an open subset $b_0 \in U \subseteq B$ such that, for all $b\in U$, the fiber $X_b = f^{-1}(b)$ is symbolically $m$-Du Bois (respectively, weakly $m$-rational).

Up to further shrinking $U$ around $b_0$, by projectivity of the map $f$, we know that $f$ is smooth-factorizable through $U \times \P^N$ for some $N\in \Z_{\geq 1}$. We also replace $X$ by $f^{-1}(U)$.

Set $Z=U\times\mathbf P^N$, choose a closed embedding $\iota\colon X \hookrightarrow Z$, and, for $b\in U$, let
$k_b\colon Z_b = \{b\}\times\mathbf P^N\hookrightarrow Z$ and
$i_b\colon X_b\hookrightarrow f^{-1}(U)$ be the fiber inclusions.

For a symbolically $m$-Du Bois fiber, let $(\cM^\bullet,F)$ be the filtered $\cD_{Z}$-module underlying $\iota_*\mathbf D(\Q_X^H)$. Then the claim follows from \propositionref{prop-HigherSingVFilt}, which implies by \propositionref{prop-VContainBC} that the base-change morphism
\[{\rm Gr}^F_p {\rm DR}_{Z_b}(k_b^!(\cM^\bullet)) \to Lk_b^! {\rm Gr}^F_p {\rm DR}_{Z/B}(\cM^\bullet)\]
is a quasi-isomorphism for all $p\leq m$. Here we have a natural quasi-isomorphism
\[ i_b^! \mathbf D(\Q_X^H) \simeq \mathbf D(\Q_{X_b}^H),\]
so the base-change morphism can be written
\[{\rm Gr}^F_p {\rm DR}_{X_b}(\mathbf D(\Q_{X_b}^H)) \to Li_b^! {\rm Gr}^F_p {\rm DR}_{X/B}(\mathbf D(\Q_X^H)).\]

We get the base-change morphism in the theorem statement by applying Grothendieck duality.

For a weakly $m$-rational fiber, let $(\cM,F)$ be the filtered $\cD_{U\times \P^N}$-module underlying $\iota_* \mathbf D({\rm IC}_X^H)$. Then for all $p\leq m$, we have the staircase-induced quasi-isomorphism
\[ {\rm Gr}^F_p {\rm DR}_{X_b}(i_b^! \mathbf D({\rm IC}_X^H))[\dim B]  \simeq {\rm Gr}^F_p {\rm DR}_{X_b} \mathbf D({\rm IC}_{X_b}^H).\]

As mentioned in Remark \ref{rmk-Staircase}, the weak $m$-rationality condition implies containment of $F_m$ in the corresponding $V^1$ pieces for a flag as in the statement of \propositionref{prop-VContainBC}. Applying \propositionref{prop-VContainBC}, we see that the base-change morphism
\[ {\rm Gr}^F_p {\rm DR}_{X_b}(i_b^! \mathbf D({\rm IC}_X^H))[\dim B] \to Li_b^! {\rm Gr}^F_p {\rm DR}_{X/B}(\mathbf D({\rm IC}_X^H))[\dim B]\]
is a quasi-isomorphism for all $p\leq m$. Composing these quasi-isomorphisms, we get 
\[ {\rm Gr}^F_p {\rm DR}_{X_b} \mathbf D({\rm IC}_{X_b}^H) \simeq Li_b^! {\rm Gr}^F_p {\rm DR}_{X/B}(\mathbf D({\rm IC}_X^H))[\dim B], \]
which induces the quasi-isomorphism claimed in the theorem statement by applying Grothendieck duality.
\end{proof}

\section{Comparisons and Open Questions} \label{sec-OpenProblems}
Though our complexes satisfy many desirable properties, several basic questions remain open.

\textit{Question 1} Can one use a co-\v{C}ech construction to define $\underline{\Omega}_{X/B}^p$ in the non-smooth-factorizable setting? If so, can one filter $\underline{\Omega}_X^j$ as in \theoremref{thm-filtration} for morphisms which are not smooth-factorizable? One issue with this is that the relation $\mathbb D_X{\rm Gr}^F_{-p} {\rm DR}_X(-) \simeq {\rm Gr}^F_p {\rm DR}(\mathbf D_X(-))$ is only known to hold for $X$ which is embeddable into a smooth variety.

For flat families over a smooth curve, Kov\'{a}cs and Taji prove functoriality of their filtered relative Du Bois complex \cite{KovacsTaji}*{Thm. 4.2}; on associated graded pieces this recovers the functoriality of Kov\'{a}cs' relative complexes. The following would be interesting, though not necessary for our purposes.

\textit{Question 2} Does our construction satisfy functoriality? Specifically, given a proper morphism $g \colon X \to Y$ of $B$-varieties (whose maps to $B$ are both smooth-factorizable) and an integer $p \in \Z$, is there a morphism
\[ \underline{\Omega}_{Y/B}^p \to Rg_* \underline{\Omega}_{X/B}^p?\]

If so, can these morphisms be constructed in such a way that the morphism
\[ \underline{\Omega}_{Y}^p \to Rg_* \underline{\Omega}_X^p\]
is $G$-filtered for all $p\in \Z$?

We turn now to the question of comparing our invariant to Kov\'{a}cs'. As pointed out to us by Kov\'{a}cs, our complex may differ from those defined by him for the following reason: in his construction, one could make a choice to either build ``bottom up'' or ``top down''. He chose ``top down'', so that the top relative Du Bois complex $\underline{\Omega}_{X/B}^r$ (where $r = \dim X - \dim B$) is equal to what is expected for a flat family $X \to B$. On the other hand, our notion seems to more closely match a ``bottom up'' procedure: we know by \theoremref{thm-filtration} that if $X\to B$ is smooth-factorizable, then $\underline{\Omega}_X^0 \to \underline{\Omega}_{X/B}^0$ is a quasi-isomorphism, since by Remark \ref{rmk-LowestPiece}, $\underline{\Omega}_{X/B}^{-1} = 0$.

Following Ning \cite{NingDeformations}, we denote Kov\'{a}cs' relative Du Bois complexes by $\underline{\Omega}_{X/B}^{p,-}$, the so-called ``left relative Du Bois complexes''.

\textit{Question 3} Can we construct a quasi-isomorphism
\[ \underline{\Omega}_{X/B}^{p,-} \simeq {\rm Gr}^F_{-p-\dim B} {\rm DR}_{X/B}(\Q_X^H)[p+\dim B] \otimes_{\cO_X} f^* \omega_B^{-1}\]
for any $p\in \Z$?

Answering this question would give a mixed Hodge module interpretation of Kov\'{a}cs' construction. The shift and twist by $f^*\omega_B^{-1}$ are motivated by the case of a smooth morphism.

Finally, Ning \cite{NingDeformations} defines, when the base $B$ is a curve, the ``right'' relative Du Bois complexes $\underline{\Omega}_{X/B}^{p,+}$. As noted there, this notion seems better suited to base change.

\textit{Question 4} For $\dim B =1$, can we construct a quasi-isomorphism
\[ \underline{\Omega}_{X/B}^{p,+} \simeq \underline{\Omega}_{X/B}^p\]
for any $p\in \Z$?

\bibliography{bibliography}

\end{document}